\documentclass{article}
\usepackage{graphicx}
\usepackage{amsmath,amsthm,amssymb,pifont,colortbl,amscd, wrapfig}
\usepackage{bigdelim,multirow}
\newtheorem{theorem}{Theorem}[section]
\newtheorem{lemma}[theorem]{Lemma}
\newtheorem{proposition}[theorem]{Proposition}
\newtheorem{example}[theorem]{Example}
\newtheorem{corollary}[theorem]{Corollary}

\theoremstyle{definition}

\newtheorem{remark}[theorem]{Remark}
\newtheorem*{acknowledgments}{Acknowledgments}

\def\Span{\mathop{\mathrm{Span}}\nolimits}
\def\Span{\mathop{\mathrm{Span}}\nolimits}

\newcommand{\cyc}[2]{ \lfloor \frac{#1}{#2}\rfloor}

\newcommand{\Z}{\mathbb{Z}[q,q^{-1}]}
\begin{document}
\title{On the colored Jones polynomials of ribbon links, boundary links and Brunnian links}
\author{Sakie Suzuki\thanks{Research Institute for Mathematical Sciences, Kyoto
University, Kyoto, 606-8502, Japan. E-mail address: \texttt{sakie@kurims.kyoto-u.ac.jp}} }

\date{November 27, 2011}
\maketitle
\begin{center}
\textbf{Abstract}
\end{center}
Habiro gave principal  ideals of $\Z$ in which certain linear combinations of the colored Jones polynomials of algebraically-split links
take values.
The author proved that  the same linear combinations for  ribbon links, boundary links and Brunnian links are contained in smaller ideals of $\Z$ generated by several elements.
In this paper, we prove that these ideals also  are principal, each generated by a product of cyclotomic polynomials.

\section{Introduction}
After the discovery of the Jones polynomial,
Reshetikhin and Turaev \cite{Re} defined  an invariant
of  framed links whose components are  colored by finite dimensional representations of  a ribbon Hopf algebra.
The \textit{colored Jones polynomial} can be defined as  the  Reshetikhin-Turaev invariant  of links  whose components are colored by finite dimensional  representations of the quantized enveloping algebra  $U_h(sl_2)$.
 
We are interested in the relationship between 
\textit{algebraic properties}  of the colored Jones polynomial and \textit{topological properties} of links. 

In this paper, we consider the following three types of links.

 A link is called  a \textit{ribbon link}  if it bounds the image of an immersion 
from a disjoint union of  disks into $S^3$ with only ribbon singularities.

An $n$-component link $L=L_1\cup\cdots \cup L_n$ is called  a \textit{boundary link}  if it bounds a disjoint union of $n$ Seifert surfaces $F_1,\ldots, F_n$ in $S^3$ such  that $L_i$ bounds $F_i$ for $i=1,\ldots,n$.

A link $L$ is called a \textit{Brunnian link} if  every proper sublink of $L$ is trivial.

In \cite{H2}, Habiro used  certain linear combinations $J_{L; \tilde P'_{l_1},\ldots , \tilde P'_{l_n}}$, $l_1,\ldots,l_n\geq 0$, of the colored Jones polynomials of a link $L$ to construct  the unified Witten-Reshetikhin-Turaev invariants for integral homology spheres.
He  proved that $J_{L; \tilde P'_{l_1},\ldots , \tilde P'_{l_n}}$ for an algebraically-split, $0$-framed link $L$ is
contained in a certain principal ideal of $\Z$ (Theorem \ref{h}).
This result was improved by the present author  \cite{sakie0, sakie1, sakie2, sakie3} in the special case of  ribbon links,
boundary links (Theorem \ref{bb}) and Brunnian links (Theorem  \ref{4}) by using ideals $I_{l_1},\ldots, I_{l_n}$ of $\Z$,
where Theorem \ref{bb} for boundary links had been conjectured by Habiro \cite{H2}.
Here,  in \cite {sakie0}, we gave an alternative proof of the fact that the Jones polynomial of  an $n$-component ribbon link is  divisible by the Jones polynomial of the  $n$-component trivial link, which was proved first by Eisermann \cite{Ei}.
The results in \cite{H2, sakie0, sakie1, sakie2,sakie3} are proved by using the \textit{universal $sl_2$ invariant of bottom tangles}  (cf. \cite{H1,H2}), which has the universality property for the colored Jones polynomial of links.

In this paper,  we prove that the ideal $I_l$, $l\geq 0$, is a principal ideal generated by a product of cyclotomic polynomials (Theorem \ref{cy}),
and   rewrite Theorems \ref{h}, \ref{bb} and \ref{4} by using these generators (Proposition \ref{cy3}).


\section{Results for the colored Jones polynomial}
In this section, we recall results in \cite{H2, sakie1, sakie2, sakie3} for the colored Jones polynomial.
For the definition of the quantized enveloping algebra $U_h(sl_2)$, see, e.g.,  \cite{Ka, H2, sakie1}.
We set $q=\exp h$.

For $m\geq 1$, let  $V_m$ denote the $m$-dimensional irreducible representation of $U_h(sl_2)$.
Let $\mathcal{R}$  denote the representation ring  of $U_h(sl_2)$ over  $\mathbb{Q}(q^{\frac{1}{2}})$, i.e.,
$\mathcal{R}$ is the $\mathbb{Q}(q^{\frac{1}{2}})$-algebra 
\begin{align*}
\mathcal{R}= \Span _{\mathbb{Q}(q^{\frac{1}{2}})}\{V_m \  | \ m\geq 1\}
\end{align*}
with the multiplication induced by the tensor product.
It is well known that $\mathcal{R}=\mathbb{Q}(q^{\frac{1}{2}})[V_2].$ 

Habiro \cite{H2}  studied  the following elements in $\mathcal{R}$
\begin{align*}
\tilde P'_l&=\frac{q^{\frac{1}{2}l}}{\{l\}_q!}\prod _{i=0}^{l-1}(V_2-q^{i+\frac{1}{2}}-q^{-i-\frac{1}{2}}),
\end{align*}
for $l\geq 0,$ which are used in an important technical step in his  construction of the unified Witten-Reshetikhin-Turaev invariants for integral homology spheres.

For the definition of the colored Jones polynomial $J_{L; X_1,\ldots , X_n}$ of $L$
with $i$th component $L_i$ colored by $X_i\in \mathcal{R}$, see,  e.g., \cite{Jo, H2, sakie1}.

Set
\begin{align*}
&\{i\}_q = q^i-1,\quad  \{i\}_{q,n} = \{i\}_q\{i-1\}_q\cdots \{i-n+1\}_q,\quad  \{n\}_q! = \{n\}_{q,n},
\end{align*}
for $i\in \mathbb{Z}, n\geq 0$.

Habiro \cite{H2} proved the following.
\begin{theorem}[Habiro \cite{H2}]\label{h}
Let $L$ be an $n$-component, algebraically-split link with  $0$-framing.
We have
\begin{align}\label{z1}
J_{L; \tilde P'_{l_1},\ldots , \tilde P'_{l_n}}\in \frac{\{ 2l_{\max}+1\}_{q, l_{\max}+1}}{\{1\} _q}\mathbb{Z}[q,q^{-1}],
\end{align} 
for $l_1,\ldots , l_n\geq 0$, where  $l_{\max}=\max (l_1\ldots, l_n)$.
\end{theorem}

Set
\begin{align*}
f_{l,k}=\{l-k\}_q!\{k\}_q!,
\end{align*}
for $0\leq k\leq l$.
For $l\geq 0$, let $I_{l}$  be the ideal of $\mathbb{Z}[q,q^{-1}]$ generated by $f_{l,0},\ldots, f_{l,l}$.

In \cite{sakie1, sakie2}, we  proved   the following.
\begin{theorem}[\cite{sakie1, sakie2}]\label{bb}
Let $L$ be an $n$-component ribbon or boundary link with $0$-framing.
For $l_1,\ldots , l_n\geq 0$, we have
\begin{align}\label{z2}
J_{L; \tilde P'_{l_1},\ldots , \tilde P'_{l_n}}\in \frac{\{ 2l_{\max}+1\}_{q, l_{\max}+1}}{\{1\} _q}\prod _{1\leq i\leq n, i\neq i_M} I_{l_i},
\end{align} 
where $l_{\max}=\max (l_1,\ldots, l_n)$ and $i_M$  is an integer such that   $l_{i_M}=l_{\max}$.
\end{theorem}
\begin{remark}
Theorem \ref{bb} for boundary links had been conjectured by Habiro \cite{H2}.
\end{remark}

In \cite{sakie3}, we prove the following.
\begin{theorem}[\cite{sakie3}]\label{4}
Let $L$ be an $n$-component Brunnian link  with $n\geq3$.
We have
\begin{align}\label{z3}
J_{L; \tilde P'_{l_1},\ldots , \tilde P'_{l_n}}&\in
\frac{\{ 2l_{\max}+1\}_{q, l_{\max}+1}}{\{1\} _q\{l_{\min}\}_q!} \prod _{1\leq i\leq n, i\neq i_M,i_m} I_{l_i},
\end{align}
for $l_1,\ldots , l_n\geq 0$, where $l_{\max}=\max (l_1,\ldots, l_n)$, $l_{\min}=\min  (l_1,\ldots, l_n)$ and $i_M,i_m$, $i_M\neq i_m$, are integers such that   $l_{i_M}=l_{\max}$, $l_{i_m}=l_{\min}$, respectively.
\end{theorem} 
Let us compare Theorems \ref{h}, \ref{bb} and \ref{4}.
For $l_1,\ldots , l_n\geq 0$, let $Z^{(l_1,\ldots, l_n)}_a$, $Z^{(l_1,\ldots, l_n)}_{r,b}$ and $Z^{(l_1,\ldots, l_n)}_{Br}$ denote the ideals of $\Z$ at the right hand sides of (\ref{z1}), (\ref{z2}) and (\ref{z3}), respectively, i.e., we set
\begin{align*}
Z^{(l_1,\ldots, l_n)}_a&=\frac{\{ 2l_{\max}+1\}_{q, l_{\max}+1}}{\{1\} _q} \Z,
\\
Z^{(l_1,\ldots, l_n)}_{r,b}&=\frac{\{ 2l_{\max}+1\}_{q, l_{\max}+1}}{\{1\} _q}\prod _{1\leq i\leq n, i\neq i_M} I_{l_i},
\\
Z^{(l_1,\ldots, l_n)}_{Br}&=\frac{\{ 2l_{\max}+1\}_{q, l_{\max}+1}}{\{1\} _q\{l_{\min}\}_q!} \prod _{1\leq i\leq n, i\neq i_M,i_m} I_{l_i}.
\end{align*}

For $l_1,\ldots , l_n\geq 0$,   we have
\begin{align*}
Z^{(l_1,\ldots, l_n)}_{r,b}\subset Z^{(l_1,\ldots, l_n)}_{a}, \quad  Z^{(l_1,\ldots, l_n)}_{r,b}\subset Z^{(l_1,\ldots, l_n)}_{Br},
\end{align*}
since we have
\begin{align*}
Z^{(l_1,\ldots, l_n)}_{r,b}=&\big(\prod _{1\leq i\leq n, i\neq i_M} I_{l_i} \big)\cdot Z^{(l_1,\ldots, l_n)}_{a}, 
\\=&\big(\{l_{\min}\}_q!I_{l_{\min}}\big)\cdot Z^{(l_1,\ldots, l_n)}_{Br}.
\end{align*}

On the other hand, there are no inclusion which satisfies  for all $l_1,\ldots , l_n\geq 0$ between  $Z^{(l_1,\ldots, l_n)}_{a}$ and $Z^{(l_1,\ldots, l_n)}_{Br}$ .
For example, we have  $Z^{(2,2,2,2)}_{a}\not \subset Z^{(2,2,2,2)}_{Br}$ and  $Z^{(2,2,2,2)}_{Br}\not \subset Z^{(2,2,2,2)}_{a}$
since
\begin{align*}
Z^{(2,2,2,2)}_{a}&=\frac{\{ 5\}_{q, 3}}{\{1\} _q} \Z
\\&=(q-1)^2(q+1)(q^2+q+1)(q^2+1)(q^4+q^3+q^2+q^1+1)\Z,
\\
 Z^{(2,2,2,2)}_{Br}&=\frac{\{ 5\}_{q, 3}}{\{1\} _q\{2\}_q!} \{1\}_q^4 \Z
 \\
 &=(q-1)^4(q^2+q+1)(q^2+1)(q^4+q^3+q^2+q^1+1)\Z.
\end{align*}
Since a  Brunnian link  with $n\geq 3$ components is algebraically-split with $0$-framing,
we have the following  refinement of  Theorem \ref{4}.
\begin{theorem}\label{41}
Let $L$ be an $n$-component Brunnian link with $n\geq3$.
We have
\begin{align*}
J_{L; \tilde P'_{l_1},\ldots , \tilde P'_{l_n}}&\in Z^{(l_1,\ldots, l_n)}_{a}\cap Z^{(l_1,\ldots, l_n)}_{Br},
\end{align*}
for $l_1,\ldots , l_n\geq 0$.
\end{theorem}

\section{Main result for the ideal $I_l$}
In this section, we state the main result of this paper.

For $l\geq 0$, recall the generators $f_{l,0},\ldots, f_{l,l}$ of the ideal $I_l$.
Set
\begin{align*}
g_l=GCD(f_{l,0},\ldots, f_{l,l}).
\end{align*}
It is clear that $I_l\subset g_l\Z$. The opposite inclusion follows if and only if  $I_l$ is principal.
Since $\Z$ is \textit{not} a principal ideal domain, it had been a problem if  $I_l$ is principal or not.
The main result in this paper (Theorem \ref{cy}) is that $I_l$ is principal,
where we determine  $g_l$ explicitly.
The proof is in  Section \ref{pr}.

For $m\geq 1$, let $\Phi _m=\prod _{d|m}(q^d-1)^{\mu (\frac{m}{d})}\in \mathbb{Z}[q]$ denote the $m$th cyclotomic polynomial, where $\prod _{d|m}$ denotes the
product  over all  positive divisors $d$ of $m$, and  $\mu$ is the M\"obius function.
 For $r\in \mathbb{Q}$, we denote by $\lfloor r \rfloor$  the largest integer smaller than  or  equal to $r$. 

\begin{theorem}\label{cy}
For $l \geq 0$, the ideal $I_l$ is the principal ideal generated by $g_l$.
Moreover, we have
\begin{align}\label{lla}
g_l&=\prod_{m\geq 1}\Phi _{m}^{t_{l,m}},
\end{align}
where
\begin{align*}
t_{l,m}&=\begin{cases} 
\cyc{l+1}{m}-1\quad \quad \text{for } 1\leq m\leq l,
\\
0\quad \quad\quad \quad \quad \ \ \text{for } l<m.
\end{cases}
\end{align*}

\end{theorem}

Here  is a table of $t_{l,m}$ for  $1\leq m \leq 4$, $0\leq l\leq 16$.
\begin{center}
\begin{tabular}{|l|r|r|r|r|r|r|r|r|r|r|r|r|r|r|r|r|r|r|}
\hline
     
     $m\setminus l$& $0$ & $1$&$2$&$3$&$4$&$5$&$6$&$7$&$8$&$9$&$10$&$11$&$12$&$13$&$14$&$15$&$16$ \\  \hline
     $1$& $0$ &$ 1$&$2$&$3$&$4$&$5$&$6$&$7$&$8$&$9$&$10$&$11$&$12$&$13$&$14$&$15$&$16 $\\\hline
    $2 $&$0$&$0 $&$ 0 $&$  1$&$1$&$2$&$2$&$3$&$3$&$4$&$4$&$5$&$5$&$6$&$6$&$7$&$7$\\ \hline
   $3 $&$0$&$0 $&$0 $&$0$&$0$&$ 1$&$1$&$1$&$2$&$2$&$2$&$3$&$3$&$3$&$4$&$4$&$4$\\ \hline
   $4 $&$0$&$ 0$&$0  $&$0 $&$0$&$ 0 $&$0 $&$1$&$1$&$1$&$1$&$2$&$2$&$2$&$2$&$3$&$3$\\
\hline
\end{tabular}
\end{center}
\begin{remark}
In \cite{sakie3},  Theorem \ref{cy}  is used in the proof of Theorem \ref{4}.
\end{remark}

Theorem \ref{cy} implies that the ideals  $Z^{(l_1,\ldots, l_n)}_{r,b}$ and $Z^{(l_1,\ldots, l_n)}_{Br}$ are principal.
Moreover, we can write a generator of each principal ideal $Z^{(l_1,\ldots, l_n)}_a$, $Z^{(l_1,\ldots, l_n)}_{r,b}$ and $Z^{(l_1,\ldots, l_n)}_{Br}$ as a product of cyclotomic polynomials as follows.

\begin{proposition}\label{cy3}
For $l_1,\ldots , l_n\geq 0$, the ideals $Z^{(l_1,\ldots, l_n)}_a$, $Z^{(l_1,\ldots, l_n)}_{r,b}$ and $Z^{(l_1,\ldots, l_n)}_{Br}$ are principal. 
Moreover, we have
\begin{align*}
 Z^{(l_1,\ldots, l_n)}_a&=\prod _{ m\geq 1}\Phi _{m}^{\cyc{2l_{\max}+1}{m}-\cyc{l_{\max}-1}{m}-\cyc{1}{m}}\Z,
\\
 Z^{(l_1,\ldots, l_n)}_{r,b}&=\prod _{1\leq m\leq 2l_{\max}+1}
 \Phi _{m}^{\cyc{2l_{\max}+1}{m}-\cyc{l_{\max}-1}{m}-\cyc{1}{m}+\sum_{1\leq i\leq n, i\not = i_M} t_{l_i, m}}\Z,
 \\
 Z^{(l_1,\ldots, l_n)}_{Br}&=\prod _{1\leq m\leq 2l_{\max}+1}
 \Phi _{m}^{\cyc{2l_{\max}+1}{m}-\cyc{l_{\max}-1}{m}-\cyc{1}{m}-\cyc{l_{\min}}{m}+\sum_{1\leq i\leq n, i\not = i_M,i_m} t_{l_i, m}}\Z.
 \end{align*}

\end{proposition}
\begin{proof}The assertion for $Z^{(l_1,\ldots, l_n)}_a$ follows from
 \begin{align}
      \{l\}_{q,i}&=\prod _{m\geq1 }\Phi _m^{\cyc{l}{m}-\cyc{l-i}{m}}, \label{cyc1}
  \end{align}
for $0\leq i\leq l$.
The assertion for $Z^{(l_1,\ldots, l_n)}_{r,b}$ and $Z^{(l_1,\ldots, l_n)}_{Br}$ follows from (\ref{cyc1}) and Theorem \ref{cy}.
\end{proof}
\begin{corollary}\label{cy4}
For $l_1,\ldots , l_n\geq 0$,  we have
\begin{align*}
 Z^{(l_1,\ldots, l_n)}_a\cap Z^{(l_1,\ldots, l_n)}_{Br}&=\prod _{ m\geq 1}\Phi _{m}^{\cyc{2l_{\max}+1}{m}-\cyc{l_{\max}-1}{m}-\cyc{1}{m}+\max (0, \sum_{1\leq i\leq n, i\not = i_M,i_m} t_{l_i, m}-\cyc{l_{\min}}{m})
}\Z.
  \end{align*}
\end{corollary}
\begin{example}
Let $L$ be an $n$-component \textit{algebraically-split} link with $0$-framing. By Theorem \ref{h} and  Proposition \ref{cy3},  we have
\begin{align*}
J_{L;\tilde P'_1,\ldots , \tilde P'_1}\in &\Phi _1\Phi _2\Phi _3\mathbb{Z}[q,q^{-1}],
\\
J_{L;\tilde P'_2,\ldots , \tilde P'_2}\in &\Phi _1^{2}\Phi _2\Phi _3\Phi _4\Phi _5\mathbb{Z}[q,q^{-1}],
\\
J_{L;\tilde P'_3,\ldots , \tilde P'_3}\in &\Phi _1^{3}\Phi _2^{2}\Phi _3\Phi _4\Phi _5\Phi _6\Phi _7\mathbb{Z}[q,q^{-1}].
\end{align*}

Let $L$ be an $n$-component \textit{Brunnian} link with $n\geq3$.
By Theorem \ref{41} and Corollary \ref{cy4}, we have
\begin{align*}
J_{L;\tilde P'_1,\ldots , \tilde P'_1}\in &\Phi _1^{n-2}\Phi _2\Phi _3\mathbb{Z}[q,q^{-1}],
\\
J_{L;\tilde P'_2,\ldots , \tilde P'_2}\in &\Phi _1^{2(n-2)}\Phi _2\Phi _3\Phi _4\Phi _5\mathbb{Z}[q,q^{-1}],
\\
J_{L;\tilde P'_3,\ldots , \tilde P'_3}\in &\Phi _1^{3(n-2)}\Phi _2^{n-1}\Phi _3\Phi _4\Phi _5\Phi _6\Phi _7\mathbb{Z}[q,q^{-1}].
\end{align*}
Let $L$ be an $n$-component \textit{ribbon} or \textit{boundary} link with $0$-framing.
By Theorem \ref{bb} and  Proposition \ref{cy3},   we have
\begin{align*}
J_{L;\tilde P'_1,\ldots , \tilde P'_1}\in &\Phi _1^{n}\Phi _2\Phi _3\mathbb{Z}[q,q^{-1}],
\\
J_{L;\tilde P'_2,\ldots , \tilde P'_2}\in &\Phi _1^{2n}\Phi _2\Phi _3\Phi _4\Phi _5\mathbb{Z}[q,q^{-1}],
\\
J_{L;\tilde P'_3,\ldots , \tilde P'_3}\in &\Phi _1^{3n}\Phi _2^{n+1}\Phi _3\Phi _4\Phi _5\Phi _6\Phi _7\mathbb{Z}[q,q^{-1}].
\end{align*}
\end{example}
\begin{example}
For $n\geq 3,$ let   $M_n$  be  Milnor's $n$-component Brunnian link  depicted in Figure \ref{fig:brunnianex}.
Note that $M_3$ is the Borromean rings. 
We have
\begin{align*}
J_{M_n;  \tilde P_1',\ldots, \tilde P_1'}&=(-1)^{n}q^{-2n+4}\Phi _1^{n-2} \Phi _2^{n-2}\Phi _3\Phi_4^{n-3},
\end{align*}
which we will prove in a forthcoming paper \cite{sakie4}.
This implies that Theorem \ref{4} is best possible for the divisibility by $\Phi_1$ and $\Phi_3$ of $J_{L;\tilde P'_1,\ldots , \tilde P'_1}$ 
with $L$ Brunnian.
By Theorem \ref{bb}, this also implies that each  $M_n$ is not ribbon or boundary.
\begin{figure}
\centering
\includegraphics[width=5cm,clip]{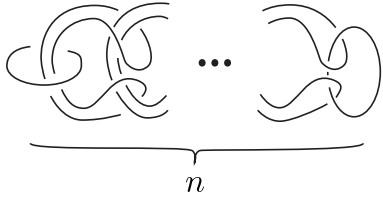}
\caption{Milnor's link $M_n$}\label{fig:brunnianex}
\end{figure}%
\end{example}
 \section{Proof of Theorem \ref{cy}}\label{pr}
We prove Theorem \ref{cy}.

For $a_1,\ldots,a_m\in \Z$, let $(a_1,\ldots,a_m)$ denote the ideal in $\Z$ generated by $a_1,\ldots,a_m\in \Z$.

For $l\geq 0$, recall that $
I_l=(f_{l,0}, f_{l,1},\ldots, f_{l,l})$
with $f_{l,i}=\{l-i\}_q!\{i\}_q!$ for $0\leq i\leq l$.

For  $0\leq k \leq l$, we have
\begin{align*}
(f_{l,0}, f_{l,1},\ldots, f_{l,k})&=\{l-k\}_q!(h_{l,k,0}, h_{l,k,1},\ldots, h_{l,k,k})
\end{align*}
with
\begin{align*}
h_{l,k,i}&=f_{l,i}/\{l-k\}_q!
\\
&=\{l-i\}_{q,k-i}\{i\}_q!, 
\end{align*}
for $ 1\leq i\leq k.$

Set 
\begin{align*}
I_{l,k}&=(h_{l,k,0}, h_{l,k,1},\ldots, h_{l,k,k}),
\\
g_{l,k}&=GCD(h_{l,k,0}, h_{l,k,1},\ldots, h_{l,k,k}).
\end{align*}
Note that $I_{l,l}=I_l$ and $g_{l,l}=g_l$.

In what follows, for $a\in \Z\setminus \{0\}$ and $m\geq1,$ let $d_m(a)$ denote the largest integer $i$ such that 
$a\in \Phi _m^i\mathbb{Z}[q,q^{-1}]$.
For  $0\leq k\leq l$, we can write
\begin{align*}
g_{l,k}&=\prod_{m\geq 1}\Phi _{m}^{d_{m}(g_{l,k})},
\end{align*}
since each $h_{l,k,i}$ is a product of cyclotomic polynomials.
\begin{lemma}\label{gcd}
For  $0\leq k \leq l$, we have
\begin{align*}
d_m(g_{l,k})&=\begin{cases}
\cyc{l+1}{m}-1- \cyc{l-k}{m}\quad \quad \text{for } 1\leq m\leq k,
\\
0\quad \quad\quad \quad \quad\quad \quad \quad \quad \text{for } k<m.
\end{cases}
\end{align*}
\end{lemma}
\begin{proof}
We have
\begin{align*}
d_m(g_{l,k})&=\min \{d_m\big(h_{l,k,i}\big) | 0\leq i\leq k \}
\\
&=\min \{d_m\big(\{l-i\}_{q,k-i}\{i\}_q!\big) | 0\leq i\leq k \}
\\
&=\min \{\cyc{l-i}{m}-\cyc{l-k}{m}+\cyc{i}{m} | 0\leq i\leq k\}
\\
&=\min \{\cyc{l-i}{m}+\cyc{i}{m} | 0\leq i\leq k\}-\cyc{l-k}{m}.
\end{align*}

If $k<m$, then we have $d_m(g_{l,k})=0$ since  $d_m\big(h_{l,k,k})=d_m\big(\{k\}_q!)=0$.

Let $1\leq m\leq k$.
Since  we have
\begin{align*}
\cyc{l-(i+am)}{m}+\cyc{i+am}{m}=\cyc{l-i}{m}+\cyc{i}{m},
\end{align*}
for $0\leq i\leq k$ and $a\in \mathbb{Z}$, we have
\begin{align*}
\min \{\cyc{l-i}{m}+\cyc{i}{m} | 0\leq i\leq k\}
&=\min \{\cyc{l-i}{m}+\cyc{i}{m} | 0\leq i\leq m-1\}.
\end{align*}
Here,  for $0\leq i\leq m-1$, we have $\cyc{i}{m}=0$ and $\cyc{l-i}{m}$ takes the minimum with $i=m-1$.
Thus we have
\begin{align*}
\min \{\cyc{l-i}{m}+\cyc{i}{m} | 0\leq i\leq m-1\}&=\cyc{l-(m-1)}{m}
\\
&=\cyc{l+1}{m}-1.
\end{align*}
This implies
\begin{align*}
d_m(g_{l,k})&=\cyc{l+1}{m}-1-\cyc{l-k}{m}.
\end{align*}
Hence we have the assertion.
\end{proof}
Note that we have the latter part (\ref{lla}) of Theorem \ref{cy} as follows.
\begin{corollary}
For $l\geq0$, we have
\begin{align*}
g_l=g_{l,l}=\prod_{m\geq 1}\Phi _{m}^{t_{l,m}}.
\end{align*}
\end{corollary}

From now, we prove the following generalization of Theorem \ref{cy}.
\begin{proposition}\label{cy2}
For $0\leq k\leq l$, the ideal $I_{l,k}$ is the principal ideal generated by $g_{l,k}$.
\end{proposition}
For $1\leq k\leq l$, set 
\begin{align*}
\tilde g_{l,k}&=g_{l,k}/g_{l,k-1}.
\end{align*}
We have
\begin{align*}
\tilde g_{l,k}&=\prod _{1\leq m\leq k}\Phi _m^{\cyc{l+1}{m}-1-\cyc{l-k}{m}-\big(\cyc{l+1}{m}-1-\cyc{l-k+1}{m}\big)}
\\
&=\prod _{1\leq m\leq k}\Phi _m^{\cyc{l-k+1}{m}-\cyc{l-k}{m}}
\\
&=\prod _{\substack{m|l-k+1 \\ 1\leq m\leq k}}\Phi _m.
\end{align*}

We use the following technical lemma. 
\begin{lemma}\label{gl}
For $1\leq k\leq l$, we have
\begin{align*}
(\{l-k+1\}_q, \{k\}_q\frac{\{k-1\}_q!}{g_{l,k-1}})=(\tilde g_{l,k}).
\end{align*}
(Note that $g_{l,k-1}=GCD(\{l\}_{q,k-1},\{l-1\}_{q,k-2}\{1\}_q, \ldots, \{k-1\}_q!)$ divides $\{k-1\}_q!$.)
\end{lemma}
\begin{proof}[Proof of Proposition \ref{cy2} by assuming Lemma \ref{gl} ]
We use induction on $k$.
For $k=0$, clearly we have 
\begin{align*}
I_{l,0}=(g_{l,0})=(\{l\}_q!).
\end{align*}

For $k\geq 1$, we have
\begin{align*}
I_{l,k}&=(h_{l,k,0}, h_{l,k,1}, \ldots, h_{l,k,k})
\\
&=(\{l\}_{q,k},\{l-1\}_{q,k-1}\{1\}_q, \ldots, \{l-k+1\}_q\{k-1\}_q!, \{k\}_q!)
\\
&=\big(\{l-k+1\}_q(\{l\}_{q,k-1},\{l-1\}_{q,k-2}\{1\}_q, \ldots, \{k-1\}_q!), \{k\}_q!\big)
\\
&=(\{l-k+1\}_q g_{l,k-1}, \{k\}_q!)
\\
&=g_{l,k-1}(\{l-k+1\}_q, \{k\}_q\frac{\{k-1\}_q!}{g_{l,k-1}})
\\
&=(g_{l,k-1}\tilde g_{l,k})
\\
&=(g_{l,k}),
\end{align*}
where the second equality is given by 
\begin{align*}
h_{l,k,0}=\{l-k+1\}_q\cdot \{l-i\}_{q,k-i-1}\{i\}_q,
\end{align*}
for $0\leq i\leq k-1$, and the third equality is given by the assumption of the induction.

Hence we have the assertion.
\end{proof} 

In what follows, we prove Lemma \ref{gl}.
We use the following two lemmas, which are well-known.
\begin{lemma}[{cf. Habiro \cite[Lemma 4.1]{H5}}]\label{hcy}
For $a,b \geq 0$, the following conditions are equivalent.
\begin{itemize}
\item[\rm{(i)}] $(\Phi_a, \Phi_b)=\Z$
\item[\rm{(ii)}] $\frac{a}{b}\not = p^{i}$ for any prime number $p \geq 0$ and $i\in \mathbb{Z}$
\end{itemize}
\end{lemma}
\begin{lemma}\label{hcy2}
Let $a_1,\ldots ,a_m, b_1,\ldots, b_n\in \Z$ such that $(a_i,b_j)=\Z$ for all $1\leq i \leq m$, $1\leq j \leq n$.
We have 
\begin{align*}
(a_1a_2\cdots a_m, b_1b_2\cdots b_n)=\Z.
\end{align*}  
\end{lemma}
\begin{proof}[Proof of Lemma \ref{gl}]
It is enough to prove the following two equalities.
\begin{align}
GCD(\{l-k+1\}_q, \{k\}_q\frac{\{k-1\}_q!}{g_{l,k-1}})&=\tilde g_{l,k}, \label{gl1}
\\
(\{l-k+1\}_q/\tilde g_{l,k}, \{k\}_q\frac{\{k-1\}_q!}{g_{l,k-1}}/\tilde g_{l,k})&=\Z,\label{gl2}
\end{align}
for $1\leq k\leq l$.

First, we prove (\ref{gl1}).
Recall that 
\begin{align}\label{rec}
\tilde g_{l,k}=\prod _{\substack{m|l-k+1 \\ 1\leq m\leq k}}\Phi _m.
\end{align}
Since $\{l-k+1\}_q= \prod _{m|l-k+1}\Phi _m$ and $d_m(\{k\}_q\frac{\{k-1\}_q!}{g_{l,k-1}})=0$ for $m>k$,
it is enough to check 
\begin{align*}
d_m(\{k\}_q\frac{\{k-1\}_q!}{g_{l,k-1}})\geq 1,
\end{align*}
 for  $m| l-k+1$, $1\leq m\leq k$.
Indeed, we have
\begin{align}
d_m(\{k\}_q)=\begin{cases}1\quad \text{for } m|k,
 \\
 0\quad \text{for } m\not|k,
 \end{cases}\label{cas}
\\
d_m(\frac{\{k-1\}_q!}{g_{l,k-1}})=\begin{cases}0\quad \text{for } m|k,
 \\
 1\quad \text{for } m\not|k.
 \end{cases}\label{cas1}
\end{align}
Here, (\ref{cas}) is clear and 
(\ref{cas1}) follows from
\begin{align*}
d_m(\frac{\{k-1\}_q!}{g_{l,k-1}})&=\cyc{k-1}{m}-t_{l,k-1,m}
\\
&=\cyc{k-1}{m}-\cyc{l+1}{m}+1+ \cyc{l-k+1}{m}
\\
&=\cyc{pm+r-1}{m}-\cyc{(p+p')m+r}{m}+1+\cyc{p'm}{m}
\\
&=\cyc{r-1}{m}+1
\\
&=\begin{cases}0\quad \text{for } r=0,
 \\
 1\quad \text{for } 1\leq r\leq m-1,
 \end{cases}
\end{align*}
where, we write  $k=mp+r$ and  $l-k+1=mp'$  with $p, p'\geq1$ and  $0\leq r\leq m-1$.

We prove (\ref{gl2}). By Lemmas \ref{hcy} and  \ref{hcy2},
it is enough to prove that there are no pair of integers $m,n\geq 1$ such that  
\begin{itemize}
\item $\frac{n}{m}=p^i$ for a prime $p$ and $i\in \mathbb{Z}$,
\item$d_m(\{l-k+1\}_q/\tilde g_{l,k})\geq 1$, and 
\item $d_n(\{k\}_q\frac{\{k-1\}_q!}{g_{l,k-1}}/\tilde g_{l,k})\geq 1$.
\end{itemize}

Note that
\begin{align*}
\{l-k+1\}_q/\tilde g_{l,k}=\prod _{\substack{m|l-k+1 \\ m>k}}\Phi _m.
\end{align*}
Let $m| l-k+1$, $m> k$. 
Recall that for $n>k$, we have $d_n(\{k\}_q\frac{\{k-1\}_q!}{g_{l,k-1}})=0$.
Assume that $1\leq n\leq k$ and $n| m$, which implies $n| l-k+1$.
The conditions $1\leq n\leq k$ and  $n| l-k+1$ imply  $d_n(\tilde g_{l,k-1})=1$ by (\ref{rec}).
By (\ref{cas}) and (\ref{cas1}), we have $d_n(\{k\}_q\frac{\{k-1\}_q!}{g_{l,k-1}})=1$.
Thus  we have  $d_n(\{k\}_q\frac{\{k-1\}_q!}{g_{l,k-1}}/\tilde{g_{l,k}})=0$,
which completes the proof.
\end{proof}

\begin{acknowledgments}
This work was partially supported by JSPS Research Fellowships for Young Scientists.
The author is deeply grateful to Professor Kazuo Habiro and Professor Tomotada Ohtsuki
for helpful advice and encouragement.
\end{acknowledgments}

\end{document}